\documentclass[12pt]{amsart}

\usepackage[latin1]{inputenc}
\usepackage{amsfonts,amsmath,amssymb,amsthm}

\newcommand{\Nn}{\mathbb{N}}

\newcommand{\Ff}{\mathbb{F}} 

\renewcommand{\epsilon}{\varepsilon}
\renewcommand{\le}{\leqslant}
\renewcommand{\ge}{\geqslant}

\newcommand{\PP}{\mathcal{P}}

{\theoremstyle{plain}
\newtheorem{theorem}{Theorem}    

\newtheorem*{theorem*}{Theorem}    

\newtheorem{lemma}[theorem]{Lemma}       
}
{\theoremstyle{remark}

}

\def\mytilde{\raisebox{-.8ex}{\~~}\hspace{-0.3em}}

\begin{document}

\title[Number of irreducible polynomials]{Number of irreducible polynomials 
in several variables over finite fields}

\author{Arnaud Bodin}
\email{Arnaud.Bodin@math.univ-lille1.fr}

\address{Laboratoire Paul Painlev\'e, Math\'ematiques, Universit\'e 
Lille 1, 59655 Villeneuve d'Ascq Cedex, France}

\date{\today}
 
\begin{abstract}
  We give a formula and an estimation for the number of irreducible
  polynomials in two (or more) variables over a finite field.
\end{abstract}

\maketitle

\section{Introduction}

Let $p$ be a prime number and $n\ge 1$.  For $q=p^n$ we denote by $\Ff_q$ the
finite field having $q$ elements.
The number of polynomials in $\Ff_q[x]$ of degree (exactly) $d$ is $N_1(d) =
q^{d+1}-q^d$. 
 The number $I_1(d)$ of irreducible polynomials of degree $d$
can be explicitly be computed with the help of the Moebius inversion formula and was already known by Gauss, 
see \cite[p.~93]{LN}.  Moreover we have an estimation for the proportion of
irreducible polynomials among all polynomials of degree $d$ (see
\cite[Ex.~26-27, p.~142]{LN}):
$$\frac{I_1(d)}{N_1(d)} \sim \frac 1 d.$$

In particular irreducible polynomials in one variable
 become more and more rare among the set
of polynomials as the degree grows.

\bigskip

Surprisingly the situation is completely different if we look at
irreducibility for polynomials in two (or more) variables.  We will prove that
most of the polynomials of degree $d$ are irreducible and we give an estimate
for this proportion as $d$ grows.

Here is the mathematical formulation : let $N_2(d)$ be the number of polynomials
in $\Ff_q[x,y]$ of degree exactly $d$ and $I_2(d)$ the number of irreducible
polynomials.

\begin{theorem*}
  \label{th:main}
$$ 1- \frac{I_2(d)}{N_2(d)} \sim \frac {q+1}{q^{d}}.$$
\end{theorem*}

In particular it implies that $\frac{I_2(d)}{N_2(d)} \to 1$ as $d \to +\infty$.

For example if $q=2$, the probability to choose an irreducible polynomial
among polynomials of degree $d$ is about $1-\frac {3}{2^d}$.  For $d=10$ we
find:
$$\frac{I_2(10)}{N_2(10)} =
\frac{73534241823793715433}{73750947497819242496}=0.997061\ldots$$ that we
approach by
$$1-\frac {3}{2^{10}} = 0.997070\ldots$$

The fact that in several variables 
almost all polynomials are irreducible is due to L.~Carlitz \cite{Ca1}. This work has been 
expanded to the study of the distribution of irreducible polynomials according not to the degree but to the 
bi-degree (where the \emph{bi-degree} of $P(x,y)$ is $(\deg_{x} P,\deg_{y} P)$)
by Carlitz  himself \cite{Ca2} and by S.~Cohen \cite{Co1} for more variables. 
More arithmetical stuff can be found in \cite{Co2}.
More recently such computations have been applied to algorithms
of factorization of multivariate polynomials, see \cite{Ra2}
and \cite{GL}.

\section{Number of polynomials}
\label{sec:Nd}

We first need to defined what is a normalized polynomial, 
let $f(x,y) \in \Ff_q[x,y]$ be a polynomial of degree exactly $d$ :
$$f(x,y) = \alpha_0x^d + \alpha_1x^dy+\alpha_2x^{d-2}y^2+\cdots + \alpha_dy^d + \text{terms of lower degree}.$$
$f$ is said to be \emph{normalized} if the first non-zero term in the sequence
$(\alpha_0,\alpha_1,\alpha_2,\ldots,\alpha_d)$ is equal to $1$.
Of course any polynomial $g$ can be written $g(x,y) = c \cdot f(x,y)$ 
where $f$ is a normalized polynomial and $c \in \Ff_q^*$.
In particular it implies that the number of normalized polynomials of degree $d$ is
the total number of polynomials of degree $d$ divided by $q-1 = \# \Ff_q^*$. 

The motivation is the following :
we will need to factorize polynomials, but unfortunately this factorization is not unique: 
for example if $g = g_1 \cdot g_2$ is the decomposition of $g \in \Ff_q[x,y]$ into a product of irreducible factors,
then $g = (c g_1) \cdot (c^{-1} g_2)$ is another factorization, for all $c\in \Ff_q^*$. This phenomenon
is problematic when we try to count the number of reducible polynomials.
However, now if $f = f_1 \cdot f_2$ is a factorization with $f, f_1,f_2$ normalized polynomials, then
this decomposition is unique (up to permutation). 

In the sequel of the text we will count normalized polynomials, normalized irreducible polynomials,...
To have the non-normalized results, just multiply by $q-1$. 
\begin{lemma}
  \label{lem:Nd}
  The number of normalized polynomials of degree exactly $d$ in $\Ff_q[x,y]$ is
$$N(d) = \left( \frac{q^{d+1}-1}{q-1} \right) \cdot q^{\frac{d(d+1)}{2}}.$$
\end{lemma}

For example $N(1) = q(q+1)$, $N(2) = \frac{q^6-q^3}{q-1}$.

\begin{proof}
The number of monomials of degree lower or equal to $d$ is
$\frac{(d+1)(d+2)}{2}$, hence the number of polynomials of degree less or
equal than $d$ is
$$N'(d) = q^\frac{(d+1)(d+2)}{2}.$$

The number of non-zero homogeneous polynomials of degree $d$ is
$$ q^{d+1}-1.$$
A polynomial of degree exactly $d$ is the sum of a non-zero homogeneous
polynomial of degree $d$ with a polynomials of degree $< d$.  Hence the number
of polynomials of degree exactly $d$ is:
$$\left( q^{d+1}-1 \right) \cdot N'(d-1).$$
To get the number of normalized polynomials we divide by $q-1$ and obtain:
$$N(d) = \left( \frac{q^{d+1}-1}{q-1} \right) \cdot N'(d-1) = \left( \frac{q^{d+1}-1}{q-1} \right) \cdot q^{\frac{d(d+1)}{2}}.$$
\end{proof}

The gap between two consecutive numbers is given by the following lemma.
\begin{lemma}
  \label{lem:Ndquot}
  \
  \begin{itemize}
  \item $\displaystyle{\frac{N(d)}{N(d+1)}=\frac{1}{q^{d+2}} \cdot \left( 1 -
        \frac{q-1}{q^{d+2}-1}\right).}$
  \item In particular $\displaystyle{\frac{N(d)}{N(d+1)} \sim
      \frac{1}{q^{d+2}}}$.
  \end{itemize}
\end{lemma}

We will need an upper bound for the product $N(a) \cdot N(b)$.
\begin{lemma}
  \label{lem:majmul}
  \
  \begin{enumerate}
  \item\label{it:majmul0} $N(a) \cdot N(b) \le N(a+b)$ for all $a\ge 1$, $b\ge
    1$;
  \item\label{it:majmul1} $N(a) \cdot N(b) \le q^3 \cdot N(a+b-1)$ for all
    $a\ge 1$, $b\ge 1$;
  \item\label{it:majmul2}$N(a) \cdot N(b) \le q^5 \cdot N(a+b-2)$ for all
    $a\ge 3$, $b\ge 3$;
  \end{enumerate}
\end{lemma}

\begin{proof}
  First of all the function defined by $M(d) = q^{d+1}-1$ verifies $M(a) \cdot
  M(b) \le q M(a+b)$ for all $a \ge 1$, $b \ge 1$.  Then
  \begin{align*}
    \frac{N(a)\cdot N(b)}{N(a+b)}
    &=   \frac{M(a)\cdot M(b)}{M(a+b)} \cdot \frac{1}{q-1} \cdot q^{\frac{a(a+1)+b(b+1)-(a+b)(a+b+1)}{2}} \\
    &\le \frac{1}{q-1} \cdot q \cdot q^{-ab} = \frac{1}{q-1} \cdot q^{-ab+1} \\
    &\le \frac{1}{q-1} \le 1. \\
  \end{align*}
  Similar calculus holds for the other bounds.
\end{proof}

\section{A formula to compute the number of irreducible polynomials}

\subsection{Notations}

We denote by $I(d)$ the number of normalized irreducible polynomials of degree exactly
$d$ and by $R(d)$ the number of normalized reducible polynomials of degree exactly $d$.
Of course we have:
$$N(d) = I(d)+R(d).$$

We will decompose the set of polynomials according to the number of
irreducible factors.  Let $S_k(d)$ be the number of normalized polynomials of degree
exactly $d$ having exactly $k$ irreducible (maybe non-distinct) factors.  Of
course
$$S_1(d) = I(d)$$
and
$$S_2(d)+\cdots+S_d(d) = R(d).$$

\subsection{Torsion product}
\label{ssec:torsion}

Let $(\ell_1,\ldots,\ell_k) \in \Nn^k$ such that

$$\underbrace{\ell_{i_1}= \cdots = \ell_{i_1+\alpha_1-1}}_{\alpha_1} <
\underbrace{\ell_{i_2}=\cdots = \ell_{i_2+\alpha_2-1}}_{\alpha_2} < \ldots <
\underbrace{\ell_{i_r}= \cdots = \ell_{k}}_{\alpha_r} $$ where $i_1=1$.

We define the following product:
$$\ell_1 \otimes \ell_2 \otimes \cdots \otimes \ell_k =
\binom{\ell_{i_1}+\alpha_1-1}{\alpha_1} \times
\binom{\ell_{i_2}+\alpha_2-1}{\alpha_2} \times \cdots \times
\binom{\ell_{i_r}+\alpha_r-1}{\alpha_r}.$$

In another language this is number of ways to choose $k$ objects from $k$ boxes (combination with
repetition), where the $i$-th box contains $\ell_i$ objects.  Moreover if
$\ell_i=\ell_j$ then boxes $i$ and $j$ contain the same objects and if
$\ell_i \neq \ell_j$ they contain no common objects.

Let us remark that:
$$\ell_1 \otimes  \cdots \otimes \ell_k  \le \ell_1 \times \cdots \times
\ell_k.$$

\subsection{Partitions}
\label{ssec:partition}

Let $\PP(k,d)$ be the set of partitions of $d$ into exactly $k$ parts:
$$\PP(k,d) = \big\lbrace [d_1,d_2,\ldots,d_k] \mid  1 \le  d_1 \le d_2 \cdots \le
d_k \text{ and } d_1+d_2+\cdots + d_k = d \big\rbrace.$$

Then the set of partitions of $d$ is:
$$\PP(d) = \PP(1,d) \cup \PP(2,d) \cup \ldots \cup \PP(d,d).$$

For example if $d=5$ we have:
$ 5 = 1+4 = 2+3 = 1+1+3 = 1+2+2 = 1+1+1+2=1+1+1+1+1$.
Then $$\PP(5)= \big\lbrace [5], [1,4], [2,3], [1,1,3], [1,2,2], [1,1,1,2], [1,1,1,1,1]  \big\rbrace.$$

Let $P(d)= \# \PP(d)$, the asymptotic behaviour of $P(d)$ is given by a
formula of Hardy and Ramanujan:
$$P(d) \sim \frac{1}{4d \sqrt 3}\exp\left(\pi \sqrt{\frac{2d}{3}}\right).$$

We will need an upper bound, \cite[p.~197]{Ay}, for all $d\ge 1$:
$$P(d) < \exp\left(\pi \sqrt{\frac{2d}{3}}\right).$$

\subsection{Formula}

\begin{lemma}
  \label{lem:sum}
$$S_k(d) = \sum_{[d_1,\ldots,d_k] \in \PP(k,d)}  I(d_1) \otimes I(d_2) \otimes  \cdots \otimes I(d_k).$$
\end{lemma}

Note that if $k\ge 2$ then all $d_i$ that appear in this formula verify $d_i
< d$.

\begin{proof}
  In fact a normalized polynomial $f$ of degree $d$ with exactly $k$ irreducible factors
  can be written $f= f_1\times \cdots \times f_k$. This decomposition is unique (up to permutation)
  if we choose the $f_i$ to be irreducible and normalized. If we denote by $d_i$ the
  degree of $f_i$ we have $d_1+\cdots+d_k = d$.  Then to a factorization we
  associate a partition $[d_1,\ldots,d_k]$ of $d$.  And the number of
  polynomials having this partition is exactly $I(d_1) \otimes \cdots \otimes
  I(d_k)$.
\end{proof}

\subsection{Algorithm}

Lemma \ref{lem:sum} provides an algorithm to compute $I(d)$ recursively.

\begin{itemize}
\item Compute $I(1)$ by hand: $I(1)=N(1) = q(q+1)$.
\item Assume that you have already computed $I(2),\ldots, I(d-1)$.
\item Calculate the sets of partitions $\PP(k,d)$, $2 \le k \le d$.
\item Apply the recursive formula
  \begin{align*}
    I(d) &= N(d) - R(d) = N(d) - \sum_{k=2}^d S_k(d) \\
    &= N(d) - \sum_{k=2}^d\sum_{[d_1,\ldots,d_k] \in \PP(k,d)} I(d_1)
    \otimes I(d_2) \otimes  \cdots \otimes I(d_k).\\
  \end{align*}
\end{itemize}

\begin{table}
  \begin{tabular}{c|c|c|c|c}
    $d$   &   $N(d)$  & $I(d)$ & $\frac{I(d)}{N(d)}$  & $1-\frac{3}{2^d}$ \\
    \hline
    1&  6 & 6 & 1 &  -0.5 \\
    2& 56 & 35 & 0.625 & 0.25  \\
    3&  960 & 694 & 0.72291\ldots & 0.625 \\
    4& 31744 & 26089 & 0.82185\ldots &  0.8125 \\
    5& 2064384 & 1862994 & 0.90244\ldots & 0.90625 \\
    6& 266338304 & 253247715 & 0.95084\ldots &  0.95312\ldots\\
    7& 68451041280 & 66799608630 & 0.97587\ldots &  0.97656\ldots\\ 
    8& {\tiny 35115652612096} &  {\tiny 34698378752226} & 0.98811\ldots  & 0.98828\ldots \\
    9& {\tiny 35993612646875136} &  {\tiny 35781375988234520} & 0.99410\ldots  &  0.99414\ldots\\
    10& {\tiny 73750947497819242496} & {\tiny 73534241823793715433} & 0.99706\ldots  &  0.99707\ldots\\
  \end{tabular}
  \ 

\caption{Number of irreducible polynomials in $\Ff_2[x,y]$.
  \label{tab:Id}}
\end{table}

Contrary to the one variable case it appears in Table \ref{tab:Id} that the
probability to choose an irreducible polynomials among polynomials of degree
$d$ tends to $1$ as $d$ tends to infinity.  Moreover the speed of this
convergence seems to be given by the formula of the introduction.

Some of these numbers appears in Sloane's \emph{Encyclopedia of Integer Sequences} \cite{Sl}, for example
the sequence $(I(d))_d = (6,35,694,\ldots)$ that gives the number of irreducible polynomials in $\Ff_2[x,y]$
is referenced as \emph{A115457}. This algorithm is implemented 
(in any number of variables and in any field) in a Maple sheet available on author's web page \cite{Bo}.

\section{Asymptotic value for the number of irreducible polynomials}

\begin{lemma}
\label{lem:majoration}
For a partition $[d_1,d_2,\ldots,d_k]\in \PP(k,d)$ not equal to $[1,d-1]$ we
have
$$I(d_1)\otimes I(d_2)\otimes \cdots \otimes I(d_k) \le q^6 \cdot N(d-2).$$
\end{lemma}

\begin{proof}
  First remember from Section \ref{ssec:torsion} that $I(d_1)\otimes \cdots
  \otimes I(d_k) \le I(d_1) \times \cdots \times I(d_k) \le N(d_1) \times
  \cdots  \times N(d_k)$.

  For the partition $[2,d-2]$ it gives
$$I(2)\otimes I(d-2) \le N(2) \cdot N(d-2) \le \frac{q^6-q^3}{q-1}\cdot N(d-2) \le q^6 \cdot N(d-2).$$

For a partition of type $[a,d-a]$, $a\ge 3$ then using Lemma
\ref{lem:majmul}-(\ref{it:majmul2}) it gives
$$I(a)\otimes I(d-a) \le N(a) \cdot N(d-a) \le q^5 \cdot N(d-2).$$

For a partition of type $[d_1,\ldots,d_k]$ with $k\ge 3$, we apply twice Lemma
\ref{lem:majmul}-(\ref{it:majmul1}) and finish using Lemma
\ref{lem:majmul}-(\ref{it:majmul0}):
\begin{align*}
  I(d_1)\otimes \cdots \otimes I(d_k)
  &\le  N(d_1) \times N(d_2) \times N(d_3) \times \cdots \times N(d_k) \\
  &\le q^3 \cdot N(d_1+d_2-1) \cdot N(d_3)  \times \cdots \times N(d_k) \\
  &\le q^3\cdot q^3 \cdot N(d_1+d_2+d_3-2) \cdot N(d_4) \times \cdots \times
  N(d_k)
  \\
  &\le q^6 \cdot N(d_1+d_2+d_3-2+d_4+\cdots +d_k) \\
  &\le q^6 \cdot N(d-2).\\
\end{align*}

\end{proof}

Lemma \ref{lem:majoration} above would enable us to prove that among reducible polynomials those 
associated to the partition $[1,d-1]$ in number $I(1)\otimes I(d-1)$ are predominant. 
This is the main idea for the proof of the next Lemma. 

\begin{lemma}
  \label{lem:Sd}
  There exists $d_0 \ge 1$ such that for all $d \ge d_0$ we have
$$1-\frac 1d \le \frac{R(d)}{N(1) \cdot N(d-1)} \le 1+\frac 1d.$$
\end{lemma}

\begin{proof}

  \emph{Upper bound.}

  $ R(d) = S_2(d)+\cdots + S_d(d)$ and each $S_k(d)$ is the sum of
  $I(d_1)\otimes \cdots \otimes I(d_k)$ over all partition $[d_1,\ldots,d_k]
  \in \PP(k,d)$.  By Lemma \ref{lem:majoration} and putting apart the
  partition $[1,d-1]$ we get that $I(d_1)\otimes \cdots \otimes I(d_k) \le q^6
  \cdot N(d-2)$.

  We recall that $P(d)$ is number of partition of $d$: $P(d) = \# \PP(d) = \#
  (\PP(1,d) \cup \ldots \cup \PP(k,d))$, see Section \ref{ssec:partition}.
  Then
  \begin{align*}
    R(d) &\le I(1)\otimes I(d-1) + P(d) \cdot q^6 \cdot N(d-2) \\
    &\le N(1) \cdot N(d-1) +  \exp\left(\pi \sqrt{\frac{2d}{3}}\right) \cdot q^6 \cdot  N(d-2).\\
  \end{align*}
  Then
  \begin{align*}
    \frac{R(d)}{N(1) \cdot N(d-1)} &\le 1+ \frac{q^6}{N(1)} \cdot \exp\left(\pi
      \sqrt{\frac{2d}{3}}\right) \cdot
    \frac{N(d-2)}{N(d-1)} \\
    &\le 1 + \frac{q^6}{q(q+1)} \cdot \exp\left(\pi \sqrt{\frac{2d}{3}}\right) \cdot \frac {1}{q^d}.\\
  \end{align*}

  Then there exists $d'_0$ such that for all $d\ge d'_0$
  \begin{equation}
    \label{eq:majSd} \tag{$*$}
    \frac{R(d)}{N(1) \cdot N(d-1)} \le 1+ \frac 1d.
  \end{equation}

  \emph{Lower bound.}

  Among reducible polynomials of degree $d$ there are polynomials of type
  $f_1\cdot f_2$ where $f_1$ is an irreducible polynomials of degree $1$ and
  $f_2$ is irreducible of degree $d-1$.  This corresponds to the partition
  $[1,d-1] \in \PP(2,d)$.  The number of polynomials corresponding to the
  partition $[1,d-1]$ is equal to $I(1)\otimes I(d-1)= N(1) \cdot I(d-1)$.

  Then for $d \ge d_0'$:
  \begin{align*}
    R(d)
    &\ge I(1)\otimes I(d-1)\\
    &= N(1) \cdot I(d-1) \\
    &= N(1) \big(N(d-1)-R(d-1)\big) \\
    &\ge N(1) \cdot \big( N(d-1) - N(1)\cdot \left(1+\frac {1}{d-1}\right) \cdot N(d-2)\big)  \qquad \text{ by (\ref{eq:majSd})} \\
    &= N(1) \cdot N(d-1)\cdot  \left( 1-N(1)\cdot\left(1+\frac
        {1}{d-1}\right)\cdot \frac{N(d-2)}{N(d-1)} \right) \\
    &\ge N(1) \cdot N(d-1) \cdot \left( 1-N(1)\cdot \left(1+\frac {1}{d-1}\right)
      \cdot \frac{1}{q^{d}} \right) \\
    &\ge N(1) \cdot N(d-1) \cdot \left( 1-\frac {1}{d}\right) \qquad \text{ $d\ge
      d_0$, for a $d_0 \ge d'_0$} \\
  \end{align*}
\end{proof}

\begin{proof}[Proof of the main theorem]
  We are now ready to prove the theorem of the introduction:
  \begin{align*}
    1-\frac{I(d)}{N(d)}
    &= \frac{N(d)-I(d)}{N(d)} =   \frac{R(d)}{N(d)} \\
    &= \frac{R(d)}{N(d-1)} \cdot \frac{N(d-1)}{N(d)} \\
    &\sim  N(1) \cdot \frac {1}{q^{d+1}} = q(q+1) \cdot \frac{1}{q^{d+1}} \\
    &\sim \frac{q+1}{q^d}.
  \end{align*}
  The first equivalence is obtained using Lemma \ref{lem:Sd} and Lemma
  \ref{lem:Ndquot}.
\end{proof}

\section{More variables}

It is not hard to extend these results to polynomials in
$\Ff_q[x_1,\ldots,x_m]$, with $m\ge 2$.  In fact only results of section
\ref{sec:Nd} have to be generalized, while the rest of the paper is still
valid.

First of all the number $N_m(d)$ of normalized polynomials of degree exactly $d$ in
$\Ff_q[x_1,\ldots,x_m]$ involves some more advanced combinatorics:
$$N_m(d) = \frac{1}{q-1} \cdot \left( q^{\binom{m+d-1}{m-1}} -1 \right) \cdot
q^{\binom{m+d-1}{m}}.$$
We get that $N_m(1)= \frac{q^{m+1}-q}{q-1}$.
Let $I_m(d)$ be the number of normalized irreducible polynomials in
$\Ff_q[x_1,\ldots,x_m]$ of degree exactly $d$ whose asymptotic behaviour
of $I_m(d)$ as $d\to +\infty$ is describe by the next result.
\begin{theorem}
$$1-\frac{I_m(d)}{N_m(d)} \sim N_m(1) \cdot  \frac{N_m(d-1)}{N_m(d)} \sim \frac{q^{m+1}-q}{q-1}
 \cdot \frac {1}{q^{\binom{m+d-1}{m-1}}}.$$
\end{theorem}
For example in $\Ff_2[x,y,z]$ the number $I_3(d)$ of irreducible polynomials
verifies:
$$1-\frac{I_3(d)}{N_3(d)} \sim \frac{14}{2^{\frac{(d+1)(d+2)}{2}}}.$$


\end{document}